\documentclass{amsart}
\usepackage[latin9]{inputenc}
\usepackage{color}
\usepackage{dsfont}
\usepackage{amsthm}
\usepackage{amssymb}
\usepackage[unicode=true,
 bookmarks=false,
 breaklinks=false,pdfborder={0 0 1},backref=section,colorlinks=true]
 {hyperref}
\hypersetup{
 linkcolor=blue,filecolor=magenta,urlcolor=cyan}

\makeatletter
\numberwithin{equation}{section}
\numberwithin{figure}{section}
\theoremstyle{plain}
\newtheorem{thm}{\protect\theoremname}
\theoremstyle{definition}
\newtheorem{defn}[thm]{\protect\definitionname}

\usepackage{graphics}

\usepackage{pdfpages}

\theoremstyle{definition}

\theoremstyle{corollary}

\theoremstyle{note}

\theoremstyle{remark}

\numberwithin{equation}{section}
\newcommand{\bs}{\backslash}

\makeatletter
\@namedef{subjclassname@2010}{%
  \textup{2010} Mathematics Subject Classification}
\makeatother

\makeatother

\providecommand{\definitionname}{Definition}
\providecommand{\theoremname}{Theorem}

\begin{document}
\title{Two more characterisations of nearly pseudocompact space }
\author{Biswajit Mitra$^1$}
\address{Department of Mathematics. The University of Burdwan, Burdwan 713104,
West Bengal, India }
\email{bmitra@math.buruniv.ac.in}
\author{Sanjib Das $^2$}
\address{Department of Mathematics. The University of Burdwan, Burdwan 713104,
West Bengal, India}
\email{ruin.sanjibdas893@gmail.com}
\thanks{$^1$Corresponding Author}
\thanks{$^2$The second author's research grant is supported by CSIR, Human Resource
Development group, New Delhi-110012, India }
\begin{abstract} In this paper we have obtained two more characterizations of nearly pseudocompact spaces.
\end{abstract}
\maketitle

\section{introduction}

In this paper by a space we always mean Tychonoff unless otherwise
mentioned. As usual $\beta X$ and $\upsilon X$ denote respectively
the Stone-$\check{C}$ech compactification and Hewitt realcompactification
of $X$. We know that the property of being pseudocompact can be described
in terms of $\beta X$ and $\upsilon X$; that is $X$ is pseudocompact
if and only if $\beta X\backslash X=$$\upsilon X\backslash X$. In
the year 1980 Henriksen and Rayburn in \cite{hr80} studied those
spaces $X$ where $\upsilon X\backslash X$ is dense in $\beta X\backslash X$
referred as nearly pseudocompact, obviously a generalization of pseudocompactness.
In that paper they have extensively studied different properties of
nearly pseudocompact space, furnishing two characterization of these
spaces. Infact they have shown that {[}\cite{hr80},Theorem 3.10{]}
$X$ is nearly pseudocompact if and only if $X$ can be expressed
as $X_{1}\cup X_{2},$ where $X_{1}$ is a regular closed almost locally
compact pseudocompact subset and $X_{2}$ is regular closed anti-locally
realcompact and $int_{X}(X_{1}\cap X_{2})=\emptyset$ if and only
if {[}\cite{hr80}, Theorem 3.18{]} every countable family of pairwise
disjoint non-empty regular hard subset of $X$ has a limit point in
$X$. Further John J. Schommer furnished few characterizations of
nearly pseudocompact spaces in his paper{[}\cite{J93}, Theorem1.3,
Theorem3.1, Corollary3.2{]}. In the year 2005 B. Mitra and S.K. Acharyya
in their paper {[}\cite{ma05}, Theorem 2.5, Theorem 3.4, Theorem
3.6, Theorem 3.7{]} produced few more characterizations of nearly
pseudocompact spaces.

In this paper we have further shown two more characterizations of
nearly pseudocompact space. In fact we have proved the following statements.
$X$ is nearly pseudocompact if and only if $X$ does not contain
any $C$-embedded copy of $\mathds{N}$ which is hard in $X$ if and
only if any family of hard sets with finite intersection property
has non-void intersection.

\section{Preliminaries}

All the basic symbols, preliminary ideas, and terminologies are taken
from the book of L.Gillman and M. Jerison, Rings of Continuous Functions
\cite{gj60}. Now we recall few notations which are used several times
over here. If $f\in C(X)$ or $C^{*}(X)$, we call the set $Z(f)=\{x\in X:f(x)=0\}$,
the zero set of $f$ and its complement is called cozero set or cozero
part of $f$ which is denoted as $cozf$. Support of a continuous
function is the closure of a set in $X$ where the function does not
vanish; that is, if $f\in C(X)$, then $cl_{X}(X\backslash Z(f))$
is the support of $f$. Let $Y\subseteq X$. If for every $f\in C(Y)$(or
$C^{*}(Y)$) there exists $g\in C(X)$ such that $g(y)=f(y)$ for
all $y\in Y$, then $Y$ is called $C$-embedded (or respectively
$C^{*}$-embedded) in $X$. When an $f\in C(X)$ is unbounded on $E$,
then $E$ has a $C$-embedded copy of $\mathds{N}$ along which $f$
goes to $\infty$ {[}\cite{gj60}, Theorem 1.21{]}. Rayburn \cite{r76} introduced the notion of hard
set in 1976.
\begin{defn}
A subspace $H$ of $X$ is called hard in $X$ if $H$ is closed in
$X\cup K$, $K=cl_{\beta X}(\upsilon X\bs X)$ where $\beta X$ and
$\upsilon X$ are the Stone-$\check{C}$ech compactification and Hewitt
realcompactification of $X$ respectively.
\end{defn}

It immediately follows that every hard set is closed in $X$, but
the converse is obviously not true. Again every compact subset of
$X$ is hard, but the converse may not be true. In the year 1980,
Henriksen and Rayburn in {[}\cite{hr80}, Theorem 3.8{]} proved the
following theorem in general set up.
\begin{thm}
The followings are equivalent.

(1) $X$ is nearly pseudocompact

(2) Every hard set is compact

(3) Every regular hard set is compact
\end{thm}

In the year 2005, Mitra and Acharyya in {[}\cite{ma05},Theorem 2.5{]}
furnished some characterizations of nearly pseudocompact space using
the notion of $C_{H}(X)$ and $H_{\infty}(X)$ among them the following
characterization is relevant to this paper.
\begin{thm}
$X$ is neraly pseudocompact if and only if $C_{H}(X)\subseteq C^{*}(X)$,
where $C_{H}(X)=\{f\in C(X)|cl_{X}(X\backslash Z(f))$ is hard in
$X\}$.
\end{thm}

In the same paper \cite{ma05} under section three, the authors introduced
an intermediate subring $\chi(X)$ of $C(X)$ to characterise nearly
pseudocompact space. The ring $\chi(X)=C^{*}(X)+C_{H}(X)$, is the
smallest subring of $C(X)$ containing $C^{*}(X)$ and $C_{H}(X)$,
which is isomorphic with $C(Z)$ for some space $Z$. In the same
paper {[}\cite{ma05},Theorem 2.8{]} it was proved that $\upsilon_{C_{H}}X=X\cup K$
and from the definition of $\chi(X)$ it follows that $\upsilon_{\chi}X=\upsilon_{C_{H}}X$
where for any subset $A$ of $C(X)$, $\upsilon_{A}X=\{p\in\beta X:f^{*}(p)\in\mathds{R},\forall f\in A\}$.
So finally we have $\upsilon_{\chi}X=X\cup K$, where $K=cl_{\beta X}(\upsilon X\backslash X)$.

The following characterizations of nearly pseudocompact space in {[}\cite{ma05},
Theorem 3.4{]} were proved using the notion of $\chi(X)$ and $\upsilon_{\chi}X$
\begin{thm}
For a space $X$, the followings are equivalent.

(1) $X$ is nearly pseudocompact.

(2) $\chi(X)=C^{*}(X)$

(3) $|\beta X\backslash\upsilon_{\chi}(X)|<2^{c}$
\end{thm}

Mitra and Das in their communicated paper \cite{MD22}
constructed nearly pseudocompact extension of a space, referred as
nearly pseudocompactification of the space, with the help of special class of $z$-ultrafilters,
called $hz$-ultrafilter.
\begin{defn}
An $hz$-filter is a $z$-filter, having a base consisting of hard-zero
sets in $X$. Maximal $hz$-filter is called $hz$-ultrafilter.
\end{defn}

Using the notion of $hz$-filters and $hz$-ultrafilters, Mitra and
Das obtained the following another characterization of nearly pseudocompact
space{[}\cite{MD22}, Theorem 2.14{]} which has been used in the present
paper.
\begin{thm}
The followings are equivalent.

(1) $X$ is nearly pseudocompact

(2) Every $hz$-filter is fixed

(3) Every $hz$-ultrafilter is fixed
\end{thm}

\section{Two more characterizations}

In this section we shall prove two more characterizations of nearly
pseudocompact space. We know from {[}\cite{gj60}, Theorem 1.21{]}
that a space $X$ is pseudocompact if and only if $X$ does not contain
any $C$-embedded copy of $\mathds{N}$. We have obtained the following
theorem for nearly pseudocompact space in the similar fashion to the
above statement for pseudocompact space.
\begin{thm}
A space $X$ is nearly pseudocompact if and only if $X$ does not
contain any $C$-embedded copy of $\mathds{N}$ which is hard in $X$.
\end{thm}

\begin{proof}
Suppose $X$ contains a $C$-embedded copy of $\mathds{N}$which is
hard in $X$. So $cl_{\beta X}\mathds{N}\backslash\mathds{N}\subseteq\beta X\backslash\upsilon_{\chi}X$.
We know $\mathds{N}$ is $C$-embedded in $X$$\Rightarrow$ $\mathds{N}$
is $C^{*}$-embedded in $X$. Therefore $cl_{\beta X}\mathds{N}=\beta\mathds{N}$
$\Rightarrow$$\beta\mathds{N}\backslash\mathds{N}\subseteq\beta X\backslash\upsilon_{\chi}X$
$\Rightarrow$ $2^{c}=|\beta\mathds{N}\backslash\mathds{N}|\leq|\beta X\backslash\upsilon_{\chi}X|$
i.e. $|\beta X\backslash\upsilon_{\chi}X|\geq2^{c}$ $\Rightarrow$$X$
is not nearly pseudocompact. Contrapositively, if $X$ is nearly pseudocompact
then $X$ does not contain a $C$-embedded copy of $\mathds{N}$which
is hard in $X$.

Conversely, suppose $X$ does not contain a $C$-embedded copy of
$\mathds{N}$which is hard in $X$. We have to show that $X$ is nearly
pseudocompact. Let $X$ is not nearly pseudocompact. Then $C_{H}(X)\nsubseteq C^{*}(X)$;
i.e. there exists a function $f\in C_{H}(X)$ but $f\notin C^{*}(X)$.
So $f$ is undounded on $cl_{X}(X\backslash Z(f))$. So there exists
a $C$-embedded copy $D$ of $\mathds{N}$ in $cl_{X}(X\backslash Z(f))$.
As $D$ is closed in $X$, it is closed in $cl_{X}(X\backslash Z(f))$.
Since $f\in C_{H}(X)$, $cl_{X}(X\backslash Z(f))$ is hard in $X$
and hence $D$ is hard in $X$, which contradicts our first assumption.
Hence $X$ is nearly pseudocompact.
\end{proof}
It is well known that a space is compact if and only if any family
of non-empty closed sets with finite intersection property has non-empty
intersection. We have proved the following similar type of result
for nearly pseudocompact space.
\begin{thm}
$X$ is nearly pseudocompact if and only if any family of non-empty
hard sets with finite intersection property has non-void intersection.
\end{thm}

\begin{proof}
Let $X$ is nearly pseudocompact. Then every hard set is compact and
any family of compact sets having finite intersection property has
non-empty intersection. Hence the result follows.

Conversely, suppose any family of hard sets with finite intersection
property has a non-void intersection. Let $\mathcal{F}$ be a $hz$-filter
in $X$. So there exists a base $\mathcal{H}$ for $\mathcal{F}$
consisting of hard sets. Since $\mathcal{F}$ is a filter , it has
finite intersection property. Then $\mathcal{H}$ has also finite
intersection property and by our assumption $\cap\mathcal{H}\neq\phi$.
Again $\cap\mathcal{H}\subseteq\cap\mathcal{F}\Rightarrow\cap\mathcal{F}\neq\phi$.
Hence $\mathcal{F}$ being arbitrary, every $hz$-filter in $X$ is
fixed. Therefore $X$ is nearly pseudocompact.
\end{proof}
Acknowledgements: The authors sincerely acknowledge the support received
from DST FIST programme (File No. SR/FST/MSII/2017/10(C))

\end{document}